\documentclass[11pt, reqno]{amsart}
\usepackage{hyperref}
\usepackage{amssymb,latexsym,amsmath,amsfonts}
\usepackage{mathrsfs}
\usepackage{graphicx}

\numberwithin{equation}{section}

\theoremstyle{definition}
\newtheorem{definition}{Definition}[section]

\theoremstyle{remark}
\newtheorem{remark}[definition]{Remark}

\theoremstyle{plain}
\newtheorem{theorem}[definition]{Theorem}
\newtheorem{result}[definition]{Result}
\newtheorem{lemma}[definition]{Lemma}
\newtheorem{proposition}[definition]{Proposition}

\newtheorem{corollary}[definition]{Corollary}

\newcommand{\R}{\mathbb{R}}

\newcommand{\Sn}{\mathbb{S}^n}

\newcommand{\distS}{d_{\mathbb S^n}}

\newcommand{\E}{E}

\makeatletter
\def\@setsubjclass{%
  \ifx\@empty\@subjclass\else
    \vspace{6pt}%
    \noindent\textit{2020 Mathematics Subject Classification.}%
    \enspace\@subjclass\par
  \fi
}
\makeatother

\begin{document}

\title{Normality Criteria for Quasimeromorphic Mappings: Avoidance and Sharing}

\author{Gopal Datt}
\address{Department of Mathematics, Babasaheb Bhimrao Ambedkar University,
Lucknow, India}
\email{ggopal.datt@gmail.com, gopal.du@gmail.com}

\author{Kushal Lalwani}
\address{Manav Rachna University, Faridabad, Haryana}
\email{kushallalwani@mru.edu.in}

\author{Ashish Kumar Trivedi}
\address{ Department of Mathematics, University of Delhi, Delhi 110007}
\email{trivediashish2016@gmail.com, aktrivedi@maths.du.ac.in}

\begin{abstract}
We establish normality criteria for families of quasimeromorphic mappings
in \(\mathbb{R}^n\), \(n\geq 2\), based on the avoidance and sharing of
continuous functions. We first establish a rescaling principle for
sequences of continuous functions and use it to derive a higher-dimensional
analogue of Lappan's avoidance criterion. We then obtain normality criteria
based on the sharing of continuous functions. In particular, we prove that
a family of \(K\)-quasimeromorphic mappings sharing \(q_0(n,K)+1\) mutually
disjoint continuous functions is normal, where \(q_0(n,K)\) is the
corresponding quasiregular Picard constant. We further consider sequences
of quasimeromorphic mappings sharing sequences of continuous functions with
a normal comparison sequence and establish a corresponding normality criterion,
together with an improved bound under an additional separation condition on the
limiting functions. Finally, we derive applications to normal
quasiregular mappings into \(\mathbb{S}^n\).
\end{abstract}

\maketitle
\section{Introduction}\label{S:intro}

The classical theory of normal families in complex analysis is notion of relative
compactness and has been an important
tool in the development function theory; see,
for example, \cite{Schiff1993}. In higher dimensions, quasiregular and
quasimeromorphic mappings provide natural counterparts of holomorphic and
meromorphic mappings, respectively, and the theory of normal families for
such mappings has been developed in various
directions; see, for example,
\cite{BBHM,Beardon2014,FletcherHahn2026,FletcherNicks2024,Makhmutov2007,Miniowitz,Vuorinen1988}.
\smallskip

One of the classical approaches to normality is through the avoidance of
sufficiently many values. Montel's theorem states that a family of
meromorphic functions which omits three distinct values is normal. In the
higher-dimensional setting, Bargmann, Bonk, Hinkkanen and Martin obtained
an analogue in terms of continuous functions. More precisely, they
established the following avoidance criterion for normality in terms of
continuous functions:
\begin{result}\cite[Theorem 1.2]{BBHM}\label{BBHM}
Let $n \ge 2$ be a positive integer and $K \in [1,\infty)$.
Then there exists an integer
$q_0 = q_0(n,K) > 0$ such that the following holds:

Let $\mathcal{F}$ be a family of $K$-quasimeromorphic mappings defined
on a domain $D \subset \mathbb{R}^n$. Suppose there exist continuous
functions
\[
h_1, \dots, h_{q_0} : D \to \mathbb{R}^n
\]
whose graphs are pairwise disjoint and such that for every
$f \in \mathcal{F}$ and every $x \in D$,
\[
f(x) \neq h_j(x) \quad \text{for all } j = 1, \dots, q_0.
\]
Then the family $\mathcal{F}$ is normal.
\end{result}

Result~\ref{BBHM} provides an normality criterion for families of
quasimeromorphic mappings in $\mathbb{R}^n$ in terms of fixed continuous
functions. The use of continuous functions in place of constant omitted
values is particularly useful when considering normality questions in
higher-dimensional settings.\smallskip

A further development of this direction was obtained by Lappan
\cite{Lappan2003}. Instead of considering fixed continuous functions,
Lappan considered sequences of continuous functions which converge locally
uniformly. In the planar meromorphic setting, he established an avoidance
criterion of the following form. Let $f_{1,m},f_{2,m}$ and $f_{3,m}$ be
sequences of continuous functions on the unit disc converging locally
uniformly to $f_1,f_2$ and $f_3$, respectively, where the limiting
functions avoid each other. If a sequence of meromorphic functions
$(g_m)$ is such that, for every $m$, the four functions
\[
g_m,\quad f_{1,m},\quad f_{2,m},\quad f_{3,m}
\]
avoid each other, then $(g_m)$ is normal. This result can be viewed as a
moving-function version of the avoidance criteria for normal families.\smallskip

In the proof of his result, Lappan also established a Rouch\'e-type
theorem involving holomorphic and continuous function. Roughly
speaking, under suitable boundary conditions, if the holomorphic function
has a zero in a disc, then the holomorphic and continuous functions must
coincide at some point in that disc.\smallskip

Normality criteria based on the sharing of values and
functions constitute an important aspect of the theory of normal
families. The classical theory of meromorphic functions has been
extensively developed in this direction, with subsequent extensions
involving shared holomorphic functions and mappings into higher-dimensional
complex spaces; see, for
example,~\cite{Schwick1992,PangZalcman2000,FangZalcman2004,
MengHu2011,DattKumar2015,DattGuptaKhannaPal2025}.\smallskip

The purpose of the present paper is to investigate a few avoidance function
techniques of normality in quasimeromorphic setting.
We first establish a rescaling result for moving continuous functions and
use it to obtain a higher-dimensional analogue of Lappan's avoidance
criterion as follows:

\begin{theorem}\label{T: QR BBHM Extension}
Consider the $q_0=q_0(n,K) \in \mathbb{N}$ sequences of continuous functions
on $D \subset \mathbb{R}^n$ say $(f_{1,n}), (f_{2,n})$,\ldots, $(f_{q_0,n})$ such that
for each $j=1,2,\ldots,q_0$ the sequence $(f_{j,n})$ converges locally uniformly
to a function on ${D}$. Suppose that these functions
$f_1, f_2, \ldots,f_{q_0}$ avoid each other on $D$.
Let $\mathcal{F} =(g_n)$ be the sequence of $K$-quasimeromorphic mappings in $D$
with the property that for each $n$,
the $q_0+1$ functions $g_n, f_{1,n}, f_{2,n},$ and $f_{q_0,n}$
avoid each other, then $\mathcal{F}$ is normal.
\end{theorem}

The proof of Theorem~\ref{T: QR BBHM Extension} combines the rescaling method for quasimeromorphic
mappings with a Rouch\'e-type theorem in the quasiregular setting and
Rickman's Picard theorem. More precisely, after rescaling a sequence of
non-normal $K$-quasimeromorphic mappings, we obtain a non-constant
quasimeromorphic mapping defined on $\mathbb{R}^n$. At the same time, the
rescaled continuous functions converge to constant functions. The
Rouch\'e-type result then enables us to transfer the avoidance property to
the limiting quasimeromorphic mapping. Since the limiting mapping omits
$q_0(n,K)$ distinct values, Rickman's Picard theorem gives a contradiction.\smallskip

We now turn our attention to the second objective of this paper:
the development of function\,--\,sharing techniques. Lets recall when we say the
two mappings share a continuous function.

\begin{definition}
Let $h:D\to\Sn$ and let $a:D\to\Sn$ be continuous. Define
\[
\E(h,a):=\{x\in D:h(x)=a(x)\}.
\]
We say that two mappings $h_1$ and $h_2$ share the continuous
function $a$ (ignoring multiplicities) if
\[
\E(h_1,a)=\E(h_2,a).
\]
\end{definition}

The rescaling argument used for the avoidance criterion
also allows us to establish the following normality criteria via sharing conditions
as follows:

\begin{theorem}\label{T: Sharing Criterion}
Let $q_0=q_0(n,K)$ be the constant given by Lemma~\ref{L:Picard}. Let
\[
\varphi_1,\ldots,\varphi_{q_0+1}
\]
be mutually disjoint continuous functions from $D$ $\subset$
$\mathbb{R}^n$.
Let $\mathcal{F}$ be a family of $K$-quasimeromorphic mappings on
$D$. If, for every pair $f,h\in\mathcal{F}$ and for each
$j=1,\ldots,q_0+1$,
\[
\E(f,\varphi_j)=\E(h,\varphi_j).
\]
then $\mathcal{F}$ is normal.
\end{theorem}

The additional function in the sharing condition is essential in the
rescaling argument: after passing to a non-constant limiting
quasimeromorphic mapping, Rickman's Picard theorem ensures that at least
two of the limiting functions are taken, while the sharing condition
forces the corresponding coincidence sets to persist under rescaling.

We further consider sharing conditions involving a normal comparison
sequence as follows:

\begin{theorem}
\label{thm:shared-moving}
Let $D\subset\R^n$ be a domain, $n\geq2$, and let $K\geq1$.
Let
$
q_0=q_0(n,K)
$
be a quasiregular Picard constant. Let $(g_n)$ be a sequence
of $K$-quasimeromorphic mappings
from $D$ into $\Sn$, and let $(h_n)$ be a normal sequence of
quasimeromorphic mappings from $D$ into $\Sn$.

For $k=1,\ldots,q$, let
$
a_{k,n}:D\longrightarrow\Sn
$
be continuous mappings satisfying
$
a_{k,n}\longrightarrow a_k
$
locally uniformly with respect to the spherical metric.

Assume that
\[
a_k(x)\neq a_l(x)
\qquad
\text{for every }x\in D,\quad k\neq l,
\]
and
\[
\E(g_n,a_{k,n})
=
\E(h_n,a_{k,n})
\]
for every $n$ and every $k=1,\ldots,q$. If
$
q\geq q_0+1,
$
then $(g_n)$ is normal in $D$.
\end{theorem}

In this setting, the continuous functions being shared are
allowed to vary, while the comparison sequence satisfies an appropriate
normality assumption. We obtain a corresponding normality criterion and
also a refinement in the case where the limiting continuous functions
remain separated from the limiting comparison function. These results
extend the avoidance approach to a broader class of sharing phenomena
for quasimeromorphic mappings.\smallskip

Finally, we consider normal quasiregular mappings. Fletcher and Nicks
introduced a framework for normality of quasiregular mappings based on
the action of a family of conformal isometries
\cite{FletcherNicks2024}. Using the avoidance and sharing criteria established
above, we obtain corresponding normality criteria for quasiregular
mappings with values in $\mathbb{S}^n$.\smallskip

The paper is organized as follows. In Section~\ref{S:Prelims}
and Section~\ref{S:Elemma}, we recall the basic
definitions and results concerning quasiregular and quasimeromorphic
mappings, normal families, and the quasiregular Picard theorem.
Section~\ref{S:Rouche-type} is devoted to the Rouch\'e-type theorem and
in Section~\ref{S: Main Result}, we establish avoidance criteria for
quasimeromorphic mappings. Section~\ref{S:SCFQM} concerns sharing of continuous
functions. Finally, Section~\ref{S: ACFQM} contains the applications to normal
quasiregular mappings with values in $\mathbb{S}^n$.

\section{Preliminaries}\label{S:Prelims}
\textbf{Notation.} Throughout this paper, we
use $\mathbb{R}^n$, $n \geq 2$ to represent the
Euclidean $n$-space, $\mathbb{S}^n$ for the unit sphere in $\mathbb{R}^{n+1}$,
identified with $\overline{\mathbb{R}^n}=\mathbb{R}^n \cup {\infty}$.\smallskip

We begin with a brief introduction to quasiregular mappings and
mention some of the important properties below.
They were initially formulated and investigated by
Yu.~G.~Reshetnyak in a sequence of papers beginning in $1966$.
In the following years, his work was significantly advanced and
expanded by O.~Martio, S.~Rickman and J.~V\"ais\"al\"a.\smallskip

\begin{definition}\label{D:Quasiregular Mapping}\cite[\S 3.1]{FletcherNicks2024}
Let $D \subset \mathbb{R}^n$, for $n\geq2$, be a domain, a \textbf{quasiregular mapping}
is a continuous mapping in the Sobolev space $W^{1, n}_{\text{loc}}(D)$
with a uniform bound on the distortion, which means that there
exists $K \geq 1$ such that
\begin{equation}\label{E: Main Inequality}
    |f'(x)|^n \leq K J_f(x)
\end{equation}
almost everywhere in $D$, where $|f'(x)|$ is the operator norm of $f'(x)$ and $J_f(x)$ is
the Jacobian determinant of $f$ at $x$.
\end{definition}
The minimal $K$ for which the inequality~\ref{E: Main Inequality} holds is called
the \emph{outer dilatation}, denoted by $K_O(f)$.
As a consequence of this there also exists
$K \geq 1$ such that
\begin{equation}\label{E: Derived Inequality}
    J_f(x) \leq K \inf_{|h|=1} |f'(x) h|^n
\end{equation}
holds almost everywhere in $D$. The minimal $K$ for which the
inequality~\ref{E: Derived Inequality}
holds is called the \emph{inner dilatation}, denoted by \textbf{$K_I(f)$}.
Next we have that
\[
    K(f) := \max(K_O(f), K_I(f))
\]
is the \emph{maximal dilatation} of $f$.\smallskip

In order to treat quasiregular mappings that are defined at the point at infinity,
or those that possess a discrete set of poles, one may pre or post compose
with an appropriate M\"obius transformation $A$
(typically chosen to be a spherical isometry) that maps the
point at infinity to the origin. This reduction allows
the standard local quasiregularity condition to be
applied in a neighbourhood of the image point.
Although such mappings are referred to
in the literature as quasimeromorphic mappings.\smallskip

Now we recall the definition of conformal metric in order to define the
normal families of quasiregular mappings.
\begin{definition}\label{D:Conformal metric}\cite[\S 2]{Beardon2014}
Let $X$ be a domain in \( \mathbb{S}^n \), for $n\geq2$.
A continuous form \( \tau(x) \)
such that \( \tau(x)|dx| \) is strictly positive called a
\textbf{conformal metric} on \( X \).
A distance function induced by a conformal metric as follows:
\[
d(u, v) = \inf_{\gamma} \int_{\gamma} \tau(x) \, |dx|
\]
where the infimum is taken over all locally rectifiable
paths \( \gamma \) joining \(
u \) to \( v \) in \( X \).
\end{definition}

Let us recall that $C(X,Y)$ denotes the collection of continuous
mappings $f : X \to Y$, where $X$ and $Y$ are subdomains of $\mathbb{S}^n$ equipped
with distance functions $d_X$ and $d_Y$ respectively arising from conformal
metrics, where continuity is with respect to these metrics.
The space $C(X,Y)$ is endowed with a distance function called the topology of
uniform convergence on compact subsets
(c.f. \cite[\S3]{Beardon2014}).

\begin{definition}\cite[Definition 3.2]{FletcherNicks2024}\label{D:K-QR mapping}
Suppose that $X$ and $Y$ are subdomains of $\mathbb{S}^n$ equipped with
conformal metrics, then
$Q(X,Y)$ the subset of $C(X,Y)$ denotes the collection
of quasiregular mappings from $X$ to $Y$.
Moreover, for $K \geq 1$, $Q_K(X,Y) \subset Q(X,Y)$
denotes the collection of $K$-quasiregular mappings from $X$ to $Y$.
\end{definition}
Reshetnyak in \cite{Reshetnyak} proved:
If $f_j : G \subset{\mathbb{R}^n} \to \mathbb{R}^n$, $j = 1,2,\ldots,$
be a sequence of $K$-quasiregular mappings converges locally uniformly
to a mapping $f$, then $f$ is also $K$-quasiregular. We can extend this result
which include point of infinity in domain or range using M\"obius maps.
One can easily see that absence of such a uniform bound on the maximal dilatation
of the mappings in the family does not always lead to
quasiregular limit function.\smallskip

We next recall the definition  of normal family in the setting
of quasiregular mappings, stated by Fletcher and Nicks in \cite{FletcherNicks2024}.
\begin{definition}\label{D:qr normal}\cite[Definition 3.3]{FletcherNicks2024}
Suppose that $X, Y$ and $Z$ with $Y \subset Z$ are subdomains
in $\mathbb{S}^n$ for $n\geq 2$
equipped with conformal metrics, let $K\geq1$ and let $\mathcal{F} \subseteq Q_{K}(X ,
Y)$, then $\mathcal{F}$ is a \textbf{normal family} if any
of the following equivalent statements hold:
  \begin{itemize}
    \item[$(i)$] $\mathcal{F}$ is a normal family relative to $\mathbb{S}^{n}$.
      Recall that the family $\mathcal{F}
     \subset C(X,Y)$ is a normal family
     relative to $Z$ if $\mathcal{F}$ is relatively compact in
     $C(X,Z)$ and the closure of $\mathcal{F}$ in
     $C(X,Z)$ has to be the closure of $\mathcal{F}$ in
     $C(X,Y)$ along with possible constant maps into boundary
     of $Y$;
    \item[$(ii)$]  $\mathcal{F}$ is relatively compact in $C(X ,
     \mathbb{S}^{n})$;
    \item[$(iii)$] for every sequence $(f_n)$ in $\mathcal{F}$ there exists
      a subsequence
     that converges locally uniformly on $X$ in the
     spherical metric to a function $f : X\rightarrow \mathbb{S}^n$.
  \end{itemize}
\end{definition}
In Definition~\ref{D:qr normal}$(i)$, if we take $X \subset \mathbb{C}$
and $Y=\mathbb{C}$
with the Euclidean metric and $Z$ is the Riemann sphere with the spherical metric,
we get the definition of \emph{normal family of meromorphic functions}
in planar domain. \smallskip

In \S~\ref{S:Rouche-type}, we prove some results for the quasiregular
mappings, for which we recall some pre-known definitions here. To begin with,
we define topological degree of mapping as
in~\cite[Ch.1, \S 4.2, p.16]{Rickman1993} as follows:\smallskip

Suppose that $f : \mathbb{S}^n \to \mathbb{S}^n$, $n \ge 1$, is a continuous mapping,
the \emph{degree} $\mu(f) \in \mathbb{Z}$ of $f$ is given by
\[
f_*(\alpha) = \mu(f)\alpha
\]
where
\[
f_* : H_n(\mathbb{S}^n) \to H_n(\mathbb{S}^n)
\]
is the induced mapping on the singular $n$-dimensional homology group
$H_n(\mathbb{S}^n)$ which is isomorphic to $\mathbb{Z}$.\smallskip

Suppose that $f : X \to \mathbb{S}^n$ is a continuous mapping on a set
$X\subset \mathbb{S}^n$ and
$U \subset X$ be open in $\mathbb{S}^n$ with $U \subset\subset X$, if
$y \in \mathbb{S}^n$ be a point with $y \notin f(\partial U)$,
then such a point $y$ is called \textbf{$(f,U)$-admissible}.
We now define the local degree or topological index
$\mu(y,f,U) \in \mathbb{Z}$ of $f$ at $y$ with respect to $U$ as follows.

Let us consider the following sequence of induced mappings on homology of pairs:
\[
H_n(\mathbb{S}^n)
\xrightarrow{\, j_* \,}
H_n\big(\mathbb{S}^n, \mathbb{S}^n \setminus (U \cap f^{-1}(y))\big)
\xleftarrow{\, e_* \,}
H_n\big(U, U \setminus f^{-1}(y)\big)
\]
\[
\xrightarrow{\, f_1* \,}
H_n(\mathbb{S}^n, \mathbb{S}^n \setminus \{y\})
\xleftarrow{\, k_* \,}
H_n(\mathbb{S}^n).
\]

where $j$, $e$ and $k$ are inclusions and $f_1$ is defined by $f$.
Then $e_*$ is an isomorphism by excision.
Since $\mathbb{S}^n \setminus \{y\}$ is homologically trivial, therefore
$k_*$ is an isomorphism.
We get a homomorphism
\[
h = k_*^{-1} f_1* e_*^{-1} j_*
: H_n(\mathbb{S}^n) \to H_n(\mathbb{S}^n)
\]
and we define $\mu(y,f,U)\ \text{by}\ h(\alpha) = \mu(y,f,U)\alpha$.
Suppose that $f : \mathbb{S}^n \to \mathbb{S}^n$ be a continuous function
and if $f^{-1}(y) \subset U$, then
$\mu(y,f,U) = \mu(f)$.\smallskip

Interested readers may also refer to \cite{Deimling1985}
for further details of topological
degree or local topological degree.
Next with the help of local topological degree, the notion of the
sense-preserving continuous map is given as follows:

\begin{definition}\cite[Ch.1, \S 4, p.17]{Rickman1993}\label{D:sense}
Suppose that $D \subset \overline{\mathbb{R}}^n$ for $n \ge 2$ be a domain
and let  $f:D \to \overline{\mathbb{R}}^n$
continuous, we say that $f$ is \textit{sense-preserving} if
$\mu(y,f,U)>0$ for all domains $U \subset\subset G$
and $y \in f(U)\setminus f(\partial U)$.
\end{definition}

\section{Auxiliary Results}\label{S:Elemma}
The following lemma tells us about the geometric properties of
non-constant quasiregular mappings.
\begin{lemma}\cite[Ch.1, Theorem 4.1, Theorem 4.5]{Rickman1993}\label{L:qr-discrete}
A non-constant quasiregular mapping is discrete, open and sense-preserving.
\end{lemma}
\noindent{Here, by a discrete quasiregular mappings we mean that the
preimage of every point consists only of isolated points.}\smallskip

Next follows the Miniowitz's version of
Zalcman's Lemma for quasiregular mapping, which characterizes the non-normality.
Roughly speaking, if a family of
quasiregular mappings is not normal at a point, then by scaling in
appropriately around that point, one can extract a sequence that converges
to a non-constant quasimeromorphic mapping defined on the whole space.

\begin{lemma}\label{L:QR Zalcman}\cite[Lemma 1]{Miniowitz}
Let $\mathcal{F}$ be the family of K-quasiregular mappings in the
unit ball $\mathbb{B}^n \subset \mathbb{R}^n$ then $\mathcal{F}$
is not normal  at $x_0 \in \mathbb{B}$ if and only if there exists,
\begin{itemize}
  \item [$(i)$] $0<r<1$,
  \item [$(ii)$] sequence of points $(x_n) \in \mathbb{B}^n$
  such that $x_n \rightarrow x_0$,
  \item [$(iii)$]sequence of mappings $(f_n) \in \mathcal{F}$,
  \item [$(iv)$] sequence of positive reals $(\rho_n)$ such
                 that $\rho_n \rightarrow 0^{+}$,
\end{itemize}
such that, $f_n(x_n+\rho_{n} \xi) \rightarrow g(\xi)$ locally uniformly
on compact subsets on $\mathbb{R}^n$ where $g$ is a non-constant
quasimeromorphic mapping $g :\mathbb{R}^n \rightarrow \mathbb{S}^n$.
\end{lemma}
\textbf{Remark.} Miniowitz's statement does not require
that the sequence of points $x_n \to x_0$, but this can be ensured easily
(see \cite[19.7.3]{IwaniecMartin2001}).\smallskip

In the classical theory of holomorphic functions in the complex plane,
Picard's theorem describes the value distribution of entire functions.
It states that if $f$ is a non-constant entire function on $\mathbb{C}$,
then $f$ assumes every complex value with at most one exception.
This result plays a fundamental role in complex analysis.
Next we state the quasiregular analogue of Picard's theorem.

\begin{lemma}\label{L:Picard}\cite[Ch.1, Theorem 2.1]{Rickman1993}
Let $n \ge 2$ and $K \ge 1$, then there exists a
$q_0 = q_0(n,K) \in \mathbb{N}$,  which depends only on $n$ and $K$ such that if
\[
f : \mathbb{R}^n \to \overline{\mathbb{R}^n} \setminus \{a_1,\dots,a_q\}
\]
is a $K$-quasimeromorphic mapping, then $f$ is constant
whenever $q \ge q_0$ and $a_1,\dots,a_q$ are distinct
points in $\overline{\mathbb{R}^n}$.
\end{lemma}
Hurwitz's theorem is another classical result concerning the behavior of zeros
under locally uniform convergence. It states that if $\{f_n\}$ is a sequence of
holomorphic functions on a domain $D \subset \mathbb{C}$ converging
locally uniformly to a holomorphic function $f$,
and if each $f_n$ has no zeros in $D$,
then either $f$ is identically zero or $f$ has no zeros in $D$.
We next state the corresponding version of Hurwitz's theorem for
quasiregular mappings.

\begin{lemma}\cite[Lemma 2]{Miniowitz}\label{L: Hurwitz}
Suppose that $D \subset \mathbb{R}^n$ for $n \geq 2$ be a domain and
let $f_m : D \to \mathbb{R}^n \setminus \{a\}$ be a sequence of
$K$-quasiregular mappings that converges locally uniformly on $D$
to a $K$-quasiregular mapping $f$. Then $f$ is either constant
or $f$ omits $a$ in $U$.
\end{lemma}

The following lemma summarizes well-known properties of the local
degree, which we are going to use in section~\ref{S:Rouche-type}.
\begin{lemma}\label{L:Topological Degree}\cite[Ch.1, Proposition 4.4]{Rickman1993}
Suppose that $X \subset \mathbb{S}^n$ for $n \geq 2$ and
$U \subset X$ be open in $\mathbb{S}^n$ with $U \subset \subset X$.
Let $f : X \to \mathbb{S}^n$ be a continuous mapping of a set $X \subset \mathbb{S}^n$,
\textit{then local degree satisfies the following:}

\begin{enumerate}
\item[$(i)$] $\mu(y,f,U)=0$ if $y \notin f(\overline{U})$ and
            the mapping $y \mapsto \mu(y,f,U)$ is constant in each component of
            $\mathbb{S}^n \setminus f(\partial U)$,

\item[$(ii)$] $\mu(y,f,U)=1$ for $y \in f(U)$, if $f$ is injective,

\item[$(iii)$] $\mu(y,f,U)=1$ for $y \in f(U)$, if $f$ is an inclusion,

\item[$(iv)$] Suppose that $f$ and $g$ are homotopic via a homotopy say $h_t$,
            $t \in [0,1]$, $h_0=f$, $h_1=g$.
            Further suppose that $y$ is $(h_t,U)$-admissible for all $t \in [0,1]$.
            Then $\mu(y,f,U)=\mu(y,g,U)$.

\end{enumerate}
\end{lemma}
Next we state another lemma, which we require to prove
Proposition~\ref{T: QR Rouche-type}.
\begin{lemma}\label{L: Cardinality}\cite[Ch.1, Proposition 4.10]{Rickman1993}
For $U \subset X$, where $X \subset \mathbb{S}^n$ for $n \geq 2$
be the domain and $y \in \mathbb{S}^n$, we write
$N(y,f,U) := \operatorname{card}\big(f^{-1}(y)\cap U\big)$.
Suppose that $f:X \to \mathbb{S}^n$ be sense-preserving, discrete and open.
Then, if $U \subset\subset X$, then
$N(y,f,U) \le \mu(y,f,U) \quad \text{for all } y \notin f(\partial U)$.
\end{lemma}
\section{Rouch\'{e}-type Theorem for quasiregular mappings}\label{S:Rouche-type}

The following result proved by Lappan in \cite{Lappan2003} was much similar to
Rouch\'{e} Theorem, here only one of the function
was assumed to be analytic, unlike the Rouch\'{e} Theorem where both the functions
are analytic.
\begin{result}\cite[Theorem 1]{Lappan2003} \label{T:Rouche-type}
Let $D_r:=\{z \in \mathbb{C}: |z| < r\}$ be the disc of radius $r>0$,
$f$ be an analytic function and $g$ be a continuous and bounded function
on the closure of disc $D_r$ such that $|f(z)| > |g(z)|$ for $|z|=r$.
If there exists $z_0 \in D_r$ such that $f(z_0)=0$, then $f$ and $g$ do not
avoid each other in $D_r$, i.e. there exists $z'$ in $D_r$ such that
$f(z')-g(z')=0$.
\end{result}

We now give an analogue of Result~\ref{T:Rouche-type} to quasiregular setting,
by taking quasiregular mapping in place of analytic function and using the techniques
from algebraic topology. We are going to use this
analogue of Rouch\'{e}\,-\,type theorem in Section~\ref{S: Main Result},
which is the special case of standard Brouwer-degree theorem.

\begin{proposition}\label{T: QR Rouche-type}
Let $B_r:=\{x: |x|<r\}$ $\subset$ $\mathbb{R}^n$
be the ball of radius $r>0$ centered at
the origin, $f$ be a quasiregular
mapping
and let $g$ be a continuous and bounded function
on the closure of ball $B_r$ such that $|f(x)| > |g(x)|$ for
$|x|=r$, i.e. $x \in \partial B_r$.
If there exists $x_0 \in B_r$ such that $f(x_0)=0$, then $f$ and $g$ do not
avoid each other in $B_r$, i.e. there exists some point $x'$ in $B_r$ such that
$f(x')-g(x')=0$.
\end{proposition}
\begin{proof}
Let us assume that there exists $x_0 \in B_r$ such that $f(x_0)=0$.\smallskip

Define $h : B_r\times I \rightarrow \mathbb{R}^n$ given by,
$$h(x,t)= f(x)-tg(x), \text{where } t\in I=[0,1] \text{ and } x \in B_r.$$
Clearly, $h$ is well-defined and continuous and
$$h(x,0)=f(x) \text{ and } h(x,1)=f(x)-g(x), \text{ where } x \in B_r .$$
Therefore, $h$ is a homotopy from $f$ to $f-g$ on $B_r$.\smallskip

It follows from Lemma~\ref{L:qr-discrete} that every non-constant
quasiregular mapping is discrete, open and sense-preserving.
Hence, the zeros of quasiregular mappings are isolated.
Let $\mu(y,f,U)$ $\in$ $\mathbb{Z}$ be the local topological degree
or topological index of $f$ at $y$ with respect to $U \subset \overline B_r$, open
and $U \subset \subset B_r$.\smallskip

Recall, for $U\subset \overline{B_r}$ and $y \in \mathbb{R}^n$,
we write $N(U,f,A)=\operatorname{card}(f^{-1}(y)\cap A)$.
Since $f$ is sense-preserving, discrete and open. Therefore, we have
$N(y,f,B_r)\leq \mu(y,f,B_r)$, for every
$y \notin f(\partial B_r)$, using Lemma~\ref{L: Cardinality}.
Therefore, if we take $U=B_r$, then
we have that $\mu(0,f,B_r) \geq 1$,
as $N(0,f,B_r)=\operatorname{card}(f^{-1}(0)\cap B_r) \geq 1$, since $f$
has a zero inside $B_r$. \smallskip

Since, $|f(x)| > |g(x)|$ for $\partial B_r$, therefore
$0 \notin f(\partial B_r)$ and we also have $f$ and $f-g$ homotopic via $h(x,t)$ and
$0 \notin h (\partial B_r)$ as $|h(x,t)| \geq |f(x)|-t|g(x)| \geq |f(x)|-|g(x)|>0$,
for all $x \in \partial B_r$,
which implies $h(x,t) \neq 0$ for all $x \in \partial B_r$.
So, $0$ is $(h,B_r)-$admissible for all $t \in [0,1]$. \smallskip

Hence, by Lemma~\ref{L:Topological Degree}($iv$) we have
$\mu(0,f-g,B_r)=\mu(0,f,B_r)\geq 1$ and using Lemma~\ref{L:Topological Degree}($i$)
there exists $x \in \bar{B_r}$, such that
$(f-g)(x)=0$, but $x \notin \partial B_r$ as
$|f(x)| > |g(x)|$ for $\partial B_r$, so $x \in B_r$.
Hence, we have that $f-g$ assumes the value zero in $B_r$.
\end{proof}

\begin{remark}
We can also observe that the roles of the quasiregular mapping and the
continuous function in Proposition~\ref{T: QR Rouche-type} cannot be reversed. To illustrate this,
consider
\[
B_1=\left\{x=(x_1,x_2)\in\mathbb{R}^2:
x_1^2+x_2^2<1\right\}.
\]
Let
\[
g(x)=\left(x_1^2+x_2^2,0\right),
\]
which is continuous on $\overline{B_1}$, and let
\[
f(x)=\left(0,\frac{1}{2}\right),
\]
which is a constant quasiregular mapping on $B_1$. Clearly,
\[
g(0,0)=(0,0).
\]
Moreover, for every $x\in\partial B_1$, we have
\[
|g(x)|=1>\frac{1}{2}=|f(x)|.
\]
However,
\[
g(x)-f(x)
=
\left(x_1^2+x_2^2,-\frac{1}{2}\right)
\neq (0,0)
\]
for every $x\in B_1$. Hence, $g-f$ does not assume the value $0$ in
$B_1$.
\end{remark}

Lappan \cite{Lappan2003} also established two corollaries of
Result~\ref{T:Rouche-type} (see results~\ref{C:Rouche-type}
and~\ref{C:Rouche-type 2} below). We obtain natural analogues of these
corollaries in the quasiregular setting (see
corollaries~\ref{C: QR Rouche-type}
and~\ref{C:QR Rouche-type 2}).
For completeness, we first state the corollaries due to Lappan,
followed by proofs of their quasiregular analogues.
We begin with the first corollary, which is essentially a
restatement of Result~\ref{T:Rouche-type} for an arbitrary disc.

\begin{result}\cite[Corollary 1]{Lappan2003}\label{C:Rouche-type}
Let $Q$ be the disc in the complex plane of radius $r>0$,
$f$ be an analytic function and $g$ be a continuous and bounded function
on the closure of disc $Q$ such that $|f(z)| > |g(z)|$ for $z \in \partial Q$.
If there exists $z_0 \in Q$ such that $f(z_0)=0$, then $f$ and $g$ do not
avoid each other in $Q$, i.e. there exists $z'$ in $Q$ such that
$f(z')-g(z')=0$.
\end{result}
So by taking motivation from above restatement of Result~\ref{T:Rouche-type},
we can also restate Proposition~\ref{T: QR Rouche-type} as follows, by replacing
the ball $B_r$ centered at origin of radius $r$, by any ball of radius
$r$ in $\mathbb{R}^n$.
\begin{corollary}\label{C: QR Rouche-type}
Let $B$ $\subset$ $\mathbb{R}^n$ be a ball of radius $r>0$, $f$ be a $K$-quasiregular
mapping and $g$ be a continuous and bounded
function on the closure of ball $B$ such that $|f(x)| > |g(x)|$ for $x \in \partial B$.
If there exists $x_0 \in B$ such that $f(x_0)=0$, then $f$ and $g$ do not
avoid each other in $B$, i.e. there exists $x'$ in $B$ such that
$f(x')-g(x')=0$.
\end{corollary}
\begin{proof}
Let us assume that $B$ is ball centered at $x_0$ and radius $r>0$,
clearly $B$ is nothing but the translated ball of $B_r=\{x \in \mathbb{R}^n: |x|<r\}$,
without loss of generality one can assume that $f$ and $g$ be the mappings
defined on $B_r$.\smallskip

Therefore, the conclusion of Corollary~\ref{C: QR Rouche-type} is evidently the
restatement of the interpretation of Proposition~\ref{T: QR Rouche-type}.
\end{proof}

Lappan also extended Result~\ref{T:Rouche-type}, replacing zero by any other
value as follows.
\begin{result}\label{C:Rouche-type 2}\cite[Corollary 2]{Lappan2003}
Let $Q$ be the disc in the complex plane of radius $r>0$,
$f$ be an analytic function and $g$ be a continuous and bounded
on the closure of disc $Q$ such that $|f(z)| > |g(z)|$ for $z \in \partial Q$.
Suppose that $0 < |\alpha| < \inf_{z \in \partial Q}(|f(z)|-|g(z)|)$ and if
there exists $z_0 \in Q$ such that $f(z_0)=0$, then there exists $z'$ in $Q$ such that
$f(z')-g(z')=\alpha$.
\end{result}

A natural extension of Proposition~\ref{T: QR Rouche-type} is quite evident
by the interpretation of Result~\ref{C:Rouche-type 2}, which could be
stated as follows.
\begin{corollary}\label{C:QR Rouche-type 2}
Let $B$ $\subset$ $\mathbb{R}^n$ be a ball of radius $r>0$, $f$ be a
$K$-quasiregular mapping
and $g$ be a continuous and bounded on the closure of ball $B$
such that $|f(x)| > |g(x)|$ for $x \in \partial B$.
Suppose that $0 < |\alpha| < \inf_{x \in \partial B}(|f(x)|-|g(x)|)$ and if
there exists $x_0 \in B$ such that $f(x_0)=0$, then there exists $x'$ in $B$ such that
$f(x')-g(x')=\alpha$.
\begin{proof}
Let if possible there exists some $x' \in \partial B$ that $|f(x')|\leq |g(x')+\alpha|$,
we get $|f(x')|\leq |g(x')|+|\alpha|$. Hence, $|f(x')-|g(x')|\leq|\alpha|$.
Which is the contradiction to our assumption,
$0 < |\alpha| < \inf_{x \in \partial B}(|f(x)|-|g(x)|)$.\smallskip

Hence, we have $|f(x)| > |g(x)+\alpha|$, for all $x \in \partial B$. Therefore,
replacing $g$ in Proposition~\ref{T: QR Rouche-type} as $g+\alpha$, will get
Corollary~\ref{C:QR Rouche-type 2} naturally.
\end{proof}
\end{corollary}
\section{Avoidance Criteria for Quasimeromorphic Mappings}\label{S: Main Result}
Lappan proved the following  result similar to
the well-known rescaling lemma by  Zalcman~\cite[\S 3(The main lemma)]{Zalcman},
to prove some of the generalised
consequences of Result~\ref{BBHM}.

\begin{result}\cite[Lemma 1]{Lappan2003}\label{L: avoid-type}
Suppose that $(f_n)$ is a sequence of continuous functions on the unit disc $\mathbb{D}
\subset \mathbb{C}$ such that, $(f_n)$ converges locally uniformly
on $\mathbb{D}$ to a function $f$. Suppose that $(z_n)$ is a sequence
in $\mathbb{D}$ converging
to a point $z_0 \in \mathbb{D}$ and there is a sequence of  positive reals say
$(\rho_n)$ converging to zero such that the sequence of functions
$(G_n(t)=g_n(z_n+\rho_n \xi))$ converges locally uniformly on $\mathbb{C}$ to
a non-constant meromorphic functions
$g: \mathbb{C} \rightarrow \mathbb{S}^2$ and if,
for each $n$, $f_n$ and $g_n$ avoid each on $\mathbb{D}$
then $g$ omits the value $f(z_0)$.
\end{result}

This raises the following question: whether an analogue of Result~\ref{L: avoid-type}
remains valid for quasiregular mappings or not?
Using our analogue of Proposition~\ref{T: QR Rouche-type}, we prove the
Proposition~\ref{L: QR avoid-type}, which is the
analogue of Result~\ref{L: avoid-type}, for
quasiregular mappings in higher dimensions, obtained by replacing sequence of
analytic function by quasiregular mappings and it is similar to that of
Miniowitz version of Zalcman re-scaling
result (see Lemma~\ref{L:QR Zalcman}).\smallskip

\begin{proposition}\label{L: QR avoid-type}
Let $(f_n)$ be a sequence of continuous functions on domain $D
\subset \mathbb{R}^n$ such that, $(f_n)$ converges locally uniformly
on $D$ to a function $f$. Suppose that $(x_n)$ is a sequence  in $D$ converging
to a point $x_0 \in D$ and there is a sequence of  positive reals say
$(\rho_n)$ converging to zero such that the sequence of functions
$(G_n(t)=g_n(x_n+\rho_n \xi))$ converges locally uniformly on $\mathbb{R}^n$ to
a non-constant quasimeromorphic mapping
$g : \mathbb{R}^n \rightarrow \mathbb{S}^n$ and if,
for each $n$, $f_n$ and $g_n$ avoid each on $\mathbb{D}$
then $g$ omits the value $f(x_0)$.
\end{proposition}
\begin{proof}
We have been given a sequence $(x_n)$ in $D$ converging
to a point $x_0 \in D$ and a sequence of  positive reals say
$(\rho_n)$ converging to zero such that the sequence of functions
$(G_n(t)=g_n(x_n+\rho_n \xi))$
converges locally uniformly on $\mathbb{R}^n$ to
a non-constant quasimeromorphic mapping
$g : \mathbb{R}^n \rightarrow \mathbb{S}^n$. \smallskip

By the use of suitable M\"obius transformation, (without loss of generality)
we may assume that $f(x_0)=0$. If $g$ omits the origin, then there is nothing to prove.
Let us assume that $g$ passes through the origin.\smallskip

Consider the functions, $h_n:\mathbb{R}^n \rightarrow \mathbb{S}^n$
$$h_n (\xi)=g_n (x_n+\rho_n \xi)-f_n (x_n +\rho_n \xi) \rightarrow g(\xi)-f(x_0)=g(\xi)$$

Since $g$ passes through the origin, or we can say $g$ assumes the value $f(x_0)=0$,
there exists a point $\xi_0$ such that $g(\xi_0)=0$.
Therefore, using the fact that non-constant quasimeromorphic mappings are discrete,
there exists $\sigma,\delta > 0$ such that $|g(\xi)|>\sigma$, whenever
$|\xi-\xi_0|=\delta$ and $g(\xi)$ is
quasiregular for $|\xi-\xi_0| \leq \delta$.\smallskip

Now for a fixed $\xi$, $x_n+\rho_n \xi \rightarrow x_0$, we have that
$f(x_n+\rho_n \xi)$ converges uniformly to origin for $|\xi-\xi_0| \leq \delta$. We also
have, $(G_n)$ converges uniformly to $g$ on $\{\xi: |\xi-\xi_0| \leq \delta\}$.\smallskip

For sufficiently large $n \in \mathbb{N}$, we have both
$|g_n(x_n+ \rho_n \xi)| > 3\sigma/4$ whenever $|\xi-\xi_0|= \delta$ as
$|g(\xi)|> \sigma$ whenever $|\xi-\xi_0|= \delta$ and $|f_n(x_n+ \rho_n \xi)|< \sigma/4$
whenever $|\xi-\xi_0|=\delta$, as $|f(x_0)|=0$. Therefore, for sufficiently large $n$
we have $|g_n(x_n+ \rho_n \xi)| \geq |f_n(x_n +\rho_n \xi)|$
whenever $|\xi-\xi_0|=\delta$.\smallskip

Furthermore, for sufficiently large $n \in \mathbb{N}$, $G_n(\xi)=g_n(x_n+\rho_n \xi)$
assumes the value zero in the disc $|\xi-\xi_0|< \delta$, otherwise by Hurwitz
Theorem for quasiregular mappings (see Lemma~\ref{L: Hurwitz}),
the limit mapping $g$ would need to
be identically zero, which is not possible as $g$ is non-constant.\smallskip

Thus, we can now apply Corollary~\ref{C: QR Rouche-type} to $g_n(z_n+ \rho_n \xi)$ and
$f_n(z_n +\rho_n \xi)$ for sufficiently large $n$, where we assume these are maps in
$\xi$ on the ball $B=\{\xi:|\xi-\xi_0|\leq \delta\}$.
Hence, we get that $g_n$ and $f_n$ do
not avoid each other on $B \subset D$, which is the contradiction to our assumption
that $g_n$ and $f_n$ avoid each other, for each $n \in \mathbb{N}$.\smallskip

Hence, we conclude that $g$ cannot pass through the origin. In particular, we have
$g$ omits the point $f(x_0)$.
\end{proof}

Using Result~\ref{L: avoid-type}, Lappan proved the following extension
of Result~\ref{BBHM}, which tends to make one think of a well-known modification
of result about the three omitted values, see \cite{Lappan1994}. \smallskip

\begin{result}\label{T: BBHM Extension}\cite[Theorem 2]{Lappan2003}
If  for each $n\in \mathbb{N}=\{1,2,3,\dots\}$,
let $f_{1,n}, f_{2,n}$ and $f_{3,n}$ be three continuous functions
on $\mathbb{D}:=\{z\in \mathbb{C} : |z| < 1\}$ such that for
each $j=1,2,3,$ the sequence $(f_{j,n})$
converges locally uniformly to a function $f_j$ on $\mathbb{D}$.
Suppose that the three functions $f_1, f_2,$ and $f_3$ avoid each other on $\mathbb{D}$.
Let $\mathcal{F} =(g_n)$ be a sequence of meromorphic functions in $\mathbb{D}$
with the property that for each $n$,
the four functions $g_n, f_{1,n}, f_{2,n},$ and $f_{3,n}$
avoid each other, then $\mathcal{F}$ is normal.
\end{result}

Now, we are going to use
Proposition~\ref{L: QR avoid-type},
to give the prove of the  Theorem~\ref{T: QR BBHM Extension}, which is a version of
Result~\ref{T: BBHM Extension} in the quasiregular
setting, by taking $q_0(n,k)$ sequences of functions, instead of three
and the proof follows the same technique as that of Lappan, recall that this
$q_0(n,k)$ is same as that of obtained in Picard's version of quasiregular
mapping (see Lemma~\ref{L:Picard}).\smallskip

\begin{proof}[Proof of Theorem~\ref{T: QR BBHM Extension}]
Suppose, on the contrary $\mathcal{F}$ is not a normal family, then $\mathcal{F}$
must not be normal in the neighbourhood of some point say $x_0 \in D$.\smallskip

According to Miniowitz version of Zalcman Lemma~\ref{L:QR Zalcman}
there exist $r \in (0, {\rm dist}(x_0, \partial D))$, a sequence $(x_n)$ of
points in ${B}(x_0, r)$ such that $(x_n)$ converges to $x_0$,
a subsequence of $(g_n)$, without loss of generality,
let us again call it $(g_{n})$
and a sequence $(\rho_n)$ in $(0,1)$, such that the sequence
$G_j(\xi)=g_{n}(x_n+\rho_n \xi)$ converges locally uniformly on $\mathbb{R}^n$
to a  a non-constant quasimeromorphic mapping
$g(\xi): \mathbb{R}^n \rightarrow \mathbb{S}^n$.\smallskip

Therefore, for each $j$  by Proposition~\ref{L: QR avoid-type},
on $f_{j,n}$ and
$g_n$, $g$ omits the $q$ distinct points
$f_1(x_0), f_2(x_0), \ldots, f_{q_0}(x_0)$, which is impossible for a non-constant
quasimeromorphic mappings.\smallskip

Thus, the assumption that $\mathcal{F}$ is not normal family is unsound.
\end{proof}
\section{Sharing Criteria for Quasimeromorphic Mappings}\label{S:SCFQM}

The passage from fixed shared values to shared functions is a natural
extension of the classical theory of normal families. In the planar
setting, Schwick~\cite{Schwick1992} established a fundamental
connection between normality and sharing values, this direction was
further developed for shared holomorphic functions.
In the quasiregular setting, normal-family criteria are
naturally formulated in terms of spherical convergence and rescaling;
see Fletcher and Nicks~\cite{FletcherNicks2024}. Motivated by these
developments, we consider shared continuous targets for
quasimeromorphic mappings. This extends the usual notion of sharing a value
in the plane to continuous functions with values in the sphere. The following elementary
separation principle allows non-coincidence to be preserved under
rescaling.\smallskip

\begin{lemma}[Spherical separation]
\label{lem:spherical-separation}
Let $K\subset\R^n$ be compact. Suppose that
$
F_n\longrightarrow b
$
uniformly on $K$ with respect to the spherical metric, where
$b\in\Sn$, and
$
A_n\longrightarrow a
$
uniformly on $K$, where $a\in\Sn$ and $a\neq b$.
Then, for all sufficiently large $n$,
\[
F_n(\xi)\neq A_n(\xi)
\qquad\text{for every }\xi\in K.
\]
\end{lemma}

\begin{proof}
Since $a\neq b$, put
\[
\delta=\distS(a,b)>0.
\]
Choose $\varepsilon>0$ such that
\[
2\varepsilon<\delta.
\]
For sufficiently large $n$,
\[
\distS(F_n(\xi),b)<\varepsilon,
\qquad
\distS(A_n(\xi),a)<\varepsilon
\]
for every $\xi\in K$. Hence
\[
\begin{aligned}
\distS(F_n(\xi),A_n(\xi))
&\geq
\distS(a,b)
-\distS(F_n(\xi),b)
-\distS(A_n(\xi),a)\\
&>\delta-2\varepsilon>0.
\end{aligned}
\]
Thus $F_n(\xi)\neq A_n(\xi)$ for all $\xi\in K$ and all sufficiently
large $n$.
\end{proof}

We first consider the case of a fixed collection of shared
continuous targets. The planar theory contains normality criteria
based on shared values and
shared holomorphic functions. Next, we are going to prove
Theorem~\ref{T: Sharing Criterion}, which
replaces these planar targets by mutually disjoint continuous
targets and uses the quasiregular Picard theorem after rescaling.

\begin{proof}[Proof of Theorem~\ref{T: Sharing Criterion}]
Assume, to the contrary, that $\mathcal{F}$ is not normal at some
$x_0\in D$. By the Miniowitz version of Zalcman's lemma
(Lemma~3.2), there exist a sequence $(f_m)\subset\mathcal{F}$,
a sequence of points $x_m\to x_0$, and a sequence of positive
numbers $\rho_m\to0$ such that
\[
g_m(y):=f_m(x_m+\rho_m y)\longrightarrow g(y)
\]
locally uniformly on $\mathbb{R}^n$, where
\[
g:\mathbb{R}^n\longrightarrow \mathbb{S}^n
\]
is a non-constant quasimeromorphic mapping.

For $j=1,\ldots,q_0+1$, define
\[
\Phi_{j,m}(y):=\varphi_j(x_m+\rho_m y).
\]
Since each $\varphi_j$ is continuous,
\[
\Phi_{j,m}(y)\longrightarrow
c_j:=\varphi_j(x_0)
\]
locally uniformly on $\mathbb{R}^n$. Since the functions
$\varphi_1,\ldots,\varphi_{q_0+1}$ are mutually disjoint, the
points
\[
c_1,\ldots,c_{q_0+1}
\]
are distinct.

By Lemma~\ref{L:Picard}, the non-constant quasimeromorphic mapping $g$ can
omit at most $q_0-1$ distinct values. Hence, among the
$q_0+1$ distinct points
$c_1,\ldots,c_{q_0+1}$, at least two must belong to the range of
$g$. Without loss of generality, suppose that
\[
g(y_1)=c_1
\qquad\text{and}\qquad
g(y_2)=c_2
\]
for some $y_1,y_2\in\mathbb{R}^n$.\smallskip

By Hurwitz's theorem for quasiregular mappings, for sufficiently
large $m$, there exist points $y_{m,1}\to y_1$ and
$y_{m,2}\to y_2$ such that
\[
g_m(y_{m,1})=\Phi_{1,m}(y_{m,1})
\quad\text{and}\quad
g_m(y_{m,2})=\Phi_{2,m}(y_{m,2}).
\]
Set
\[
z_{m,1}=x_m+\rho_m y_{m,1},
\qquad
z_{m,2}=x_m+\rho_m y_{m,2}.
\]
Then $z_{m,1},z_{m,2}\to x_0$ and
\[
f_m(z_{m,1})=\varphi_1(z_{m,1}),
\qquad
f_m(z_{m,2})=\varphi_2(z_{m,2}).
\]

Fix any $h\in\mathcal{F}$. By the shared-target condition,
\[
h(z_{m,1})=\varphi_1(z_{m,1})
\quad\text{and}\quad
h(z_{m,2})=\varphi_2(z_{m,2}).
\]
Letting $m\to\infty$ and using the continuity of $h$ and
$\varphi_j$, we obtain
\[
h(x_0)=c_1
\quad\text{and}\quad
h(x_0)=c_2.
\]
Thus $c_1=c_2$, contradicting the fact that
$c_1$ and $c_2$ are distinct. Therefore, $\mathcal{F}$ is normal.
\end{proof}

Next we prove Theorem~\ref{thm:shared-moving},
which allows the shared targets to vary with the sequence. Such
moving-target formulations are closely related to planar domain and higher-dimensional
normality criteria,
see~\cite{DattGuptaKhannaPal2025, JiangGao2011}. The following theorem
replaces this framework by continuous moving
targets for quasimeromorphic mappings.

\begin{proof}[Proof of Theorem~\ref{thm:shared-moving}]
Suppose, to the contrary, that $(g_n)$ is not normal in $D$.
By the rescaling theorem, after passing to a subsequence there
exist $x_n\to x_0\in D$ and $\rho_n\to0^+$ such that
$
G_n(\xi):=
g_n(x_n+\rho_n\xi)
$
converges locally uniformly on $\R^n$ to a nonconstant
$K$-quasiregular mapping
$
G:\R^n\longrightarrow\Sn.
$.\smallskip

Since $(h_n)$ is normal, after passing to a further subsequence,
we may assume
$
h_n\longrightarrow h
$
locally uniformly in the spherical metric on $D$.\smallskip

For each $k$, define
\[
A_{k,n}(\xi)
:=
a_{k,n}(x_n+\rho_n\xi).
\]
Then
\[
A_{k,n}(\xi)\longrightarrow a_k(x_0)
\]
locally uniformly on $\R^n$.

Also,
\[
h_n(x_n+\rho_n\xi)
\longrightarrow h(x_0)
\]
locally uniformly on $\R^n$.

Since
\[
a_1(x_0),\ldots,a_q(x_0)
\]
are pairwise distinct, at most one of them can equal $h(x_0)$.
Hence there is a subset
\[
I\subset\{1,\ldots,q\}
\]
with
\[
|I|\geq q-1
\]
such that
\[
a_k(x_0)\neq h(x_0)
\qquad(k\in I).
\]

Fix $k\in I$. By Lemma~\ref{lem:spherical-separation}, on every
compact subset of $\R^n$,
\[
h_n(x_n+\rho_n\xi)
\neq
a_{k,n}(x_n+\rho_n\xi)
\]
for all sufficiently large $n$.

Therefore
\[
x_n+\rho_n\xi
\notin
\E(h_n,a_{k,n})
\]
on that compact set. Since
\[
\E(g_n,a_{k,n})
=
\E(h_n,a_{k,n}),
\]
we also have
\[
g_n(x_n+\rho_n\xi)
\neq
a_{k,n}(x_n+\rho_n\xi)
\]
for all sufficiently large $n$.

Applying Proposition~\ref{L: QR avoid-type} gives
\[
G(\xi)\neq a_k(x_0)
\qquad
\text{for every }\xi\in\R^n.
\]

Thus $G$ omits the distinct values
$
\{a_k(x_0):k\in I\}.
$ Since,
$
|I|\geq q-1\geq q_0,
$
the mapping $G$ omits at least $q_0$ distinct points of $\Sn$,
contradicting the choice of $q_0$.
Therefore $(g_n)$ is normal in $D$.
\end{proof}

The preceding theorem loses at most one target because one limiting
target may coincide with the limit of the normal comparison sequence.
If such a coincidence is excluded, no target is lost in the rescaling
argument, yielding the following improvement.

\begin{corollary}
\label{cor:shared-separated}
Under the hypotheses of Theorem~\ref{thm:shared-moving}, suppose
in addition that
\[
a_k(x)\neq h(x)
\qquad
\text{for every }x\in D
\]
and every $k=1,\ldots,q$.
Then $q\geq q_0$
is sufficient to guarantee that $(g_n)$ is normal in $D$.
\end{corollary}

\begin{proof}
Suppose that $(g_n)$ is not normal. As in the proof of
Theorem~\ref{thm:shared-moving}, we obtain a nonconstant
$K$-quasiregular mapping
\[
G:\R^n\longrightarrow\Sn
\]
and
\[
h_n(x_n+\rho_n\xi)\longrightarrow h(x_0).
\]

The additional hypothesis gives
\[
a_k(x_0)\neq h(x_0)
\]
for every $k=1,\ldots,q$.

Therefore Lemma~\ref{lem:spherical-separation} and the sharing
condition imply, after rescaling, that
\[
g_n(x_n+\rho_n\xi)
\neq
a_{k,n}(x_n+\rho_n\xi)
\]
on every compact subset for all sufficiently large $n$.

Proposition~\ref{L: QR avoid-type} consequently gives
\[
G(\xi)\neq a_k(x_0)
\]
for every $\xi\in\R^n$ and every $k=1,\ldots,q$.

Thus $G$ omits $q\geq q_0$ distinct points of $\Sn$, contradicting
the choice of $q_0$.

Hence $(g_n)$ is normal in $D$.
\end{proof}

The counting in Theorem~\ref{thm:shared-moving} is determined by the
possible coincidence of one limiting target with the limiting
comparison value.

\begin{remark}[Counting of the shared targets]
Among the $q$ distinct limiting targets
\[
a_1(x_0),\ldots,a_q(x_0),
\]
at most one can coincide with the limiting value $h(x_0)$.
Hence at least $q-1$ targets are omitted by the rescaled limit.
Consequently,
\[
q-1\geq q_0
\]
is sufficient, and hence
\[
q\geq q_0+1.
\]
If
\[
a_k(x_0)\neq h(x_0)
\qquad(k=1,\ldots,q),
\]
no target is lost and
\[
q\geq q_0
\]
suffices.
\end{remark}

Finally, the normality assumption on the comparison sequence is
structural. It provides the limiting value which separates the
moving targets from the comparison sequence after rescaling.\smallskip

\begin{remark}[Role of the normal comparison sequence]
The normality of $(h_n)$ is essential. Without it, the sharing
condition alone cannot force normality. Indeed, if one takes
$h_n=g_n$ for a non-normal sequence $(g_n)$, then
\[
\E(g_n,a_{k,n})=\E(h_n,a_{k,n})
\]
holds trivially.
\end{remark}
\section{Application of our main results to Normal Quasiregular Mappings}\label{S: ACFQM}
We begin this section by briefly introducing the concept of normal functions.
The notion of normality, introduced by Montel, was later extended to individual
meromorphic functions by Yosida~\cite{Yosida1934},
who defined \emph{property $(A)$} via the normality of the family of
translates $\{f(z+a_i)\}$, for given any sequence of complex number
$(a_{i})$. Functions satisfying this condition were said to belong to \emph{class $(A)$}.
This idea was further developed by Noshiro~\cite{Noshiro1938},
who studied meromorphic functions in the unit disc $\mathbb{D}$
and associated them with the family
\[
f_a(z) = f\!\left(\frac{z-a}{\overline{a}z - 1}\right), \quad a \in \mathbb{D},
\]
defining class $(A)$ through the normality of $\{f_a\}$ in $\mathbb{D}$.\smallskip

The name \textbf{normal function} was given later by
Lehto and Virtanen~\cite{LehtoVirtanen1957} in $1957$, where they introduced the
general notion of normal meromorphic functions as follows:
A meromorphic function $f$ in a hyperbolic domain
$D \subset \mathbb{C}$ is said to be normal if the
family of its compositions with automorphisms of $D$, namely
\[
\{ f \circ \phi : \phi \in \mathrm{Aut}(D) \},
\]
forms a normal family in $D$ with respect to the spherical metric.\smallskip

These functions exhibit several features analogous to those of bounded analytic
functions, with invariance under M\"obius transformations playing a central role
in their analysis. For a more comprehensive study of normal meromorphic
functions, we refer the reader to~\cite{Pommerenke1975}.\smallskip

Motivated by the corresponding notion for meromorphic function in the planar setting,
Fletcher and Nicks
in~\cite{FletcherNicks2024}, introduced the following
definition of normal quasiregular mappings.
\begin{definition}\cite[Definition $4.1$]{FletcherNicks2024}\label{D: NQM}
For $n \geq 2$, suppose $X \subset \mathbb{S}^n$ is a domain
equipped with a metric $d_X$ induced by a conformal metric.
Suppose that $\mathcal{G}$ is a family of transitive, conformal,
orientation-preserving and surjective self isometries of $X$.
A quasiregular mapping $f : X \to \mathbb{S}^n$
is said to be \emph{normal} if the associated family
\[
\mathcal{F} = \{ f \circ A : A \in \mathcal{G} \} \subset Q_K(X, \mathbb{S}^n)
\]
forms a normal family.

\end{definition}

Lappan in~\cite{Lappan2003}, proved the following version of Result~\ref{BBHM}
by Bargmann et al. for normal functions in the planar domain.
\begin{result}\cite[Theorem 3]{Lappan2003}\label{T: BBHM NF}
Suppose that $f_1, f_2,$ and $f_3$ be three continuous functions
on a unit disc $\mathbb{D}$ avoiding each other uniformly
on $\mathbb{D}$, such that the family
$\mathcal{F}_j =\{ f_j \circ \varphi : \varphi \in \mathrm{Aut}(\mathbb{D}) \}$
is normal in $\mathbb{D}$ for each $j=1,2$ and $3$.
Let $g$ be a meromorphic function in $D$
with the property that the four functions $g, f_1, f_2,$ and $f_3$ avoid each other on
$\mathbb{D}$, then $g$ is a normal function.
\end{result}

Now one can ask a natural question: whether an analogue of
Result~\ref{T: BBHM NF} can be given in quasiregular setting or not? Next,
we are going to prove the Theorem~\ref{T: ACNQM1},
which answers this question affirmatively that one
can obtain the analogue of Result~\ref{BBHM}
for normal quasiregular mapping. First, we will prove the analogue for the
quasiregular mappings from $X$ into $\mathbb{S}^n$ as follows. \smallskip

\begin{theorem}\label{T: ACNQM1}
For $n\geq2$, let $X \subset \mathbb{S}^n$ be a domain, equipped with
the conformal metric $d_X$ and let $\mathcal{G}$ be a family of transitive, conformal,
orientation-preserving and surjective self isometries of $X$. Suppose that
$f_1, f_2, \dots$, $f_{q_0}$ be $q_0=q_0(n,K)$
continuous functions from $X$ into $\mathbb{S}^n$, avoiding each other uniformly
on $X$, such that the family
$\mathcal{F}_j =\{ f_j \circ A : A \in \mathcal{G}\} \subset Q_K(X, \mathbb{S}^n)$
is normal for each $j=1,2, \dots$, $q_0$.
Let $g: X \rightarrow \mathbb{S}^n$ be a $K$-quasiregular mapping,
with the property that these $q_0+1$ functions $g, f_1, f_2,\dots$, $f_{q_0}$
avoid each other on $X$, then $g$ is a normal quasiregular mapping.
\end{theorem}
\begin{proof}
Given that the family
\[
\mathcal{F}_j=\{ f_j \circ A : A \in \mathcal{G} \}
\]
is normal for each $j=1,2 \dots,q_0$.\smallskip

Let $(A_n)$ be an arbitrary sequence
in $\mathcal{G}$. Since, the family $\mathcal{F}_j$ is normal for
each $j=1,2,\dots, q_0$, therefore there exists a subsequence of $(A_n)$,
which we still denote by $(A_n)$, such that the sequence
$(f_j \circ A_n)$ converges uniformly on compact subsets
of $X$ to a limit function $F_j$, for each $j=1,2,\dots, q_0$.\smallskip

Since, each $f_j$ avoids each other \emph{uniformly} on $X$, so
the limit functions $F_j$ also avoid each other on $X$,
for $j=1,2,\dots, q_0$.\smallskip

Also the functions $f_j \circ A_n$ avoid each other on $X$ for each $n$.\smallskip

Consider the sequence, $(g_n)=(g \circ A_n)$ of $K$-quasiregular mappings,
since  $g, f_1, f_2,\dots$, $f_{q_0}$
avoid each other on $X$, so $g_n, f_1 \circ A_n,
f_2 \circ A_n,\dots$, $f_{q_0} \circ A_n$ also avoid each other on $X$ for each $n$.
Hence, it follows from
Theorem~\ref{T: QR BBHM Extension}, that $(g \circ A_n)$ is normal.\smallskip

Since $A_n$ was the arbitrary sequence in $\mathcal{G}$,
it follows from Definition~\ref{D: NQM} that $g$ is
normal quasiregular mapping.
\end{proof}

\textbf{Remark:}
We note that the condition requiring
the functions $f_1, f_2,\dots$, $f_{q_0}$ to avoid each other uniformly on $D$
is essential. Without this assumption, it is possible that the
corresponding limit functions may coincide at certain points.

\subsection{Sharing Criterion For Normal mappings into $\mathbb{S}^n$}
\begin{theorem}
\label{T: sharing-NQM-S}
Let $X,\mathcal G$ and $q_0=q_0(n,K)$ be as in
Theorem~\ref{T: ACNQM1}. Let
\[
a_1,\ldots,a_q:X\longrightarrow\mathbb S^n
\]
be continuous mappings satisfying
\[
\distS(a_j(x),a_k(x))\geq\delta>0
\qquad
(x\in X,\ j\neq k).
\]
Assume that each family
\[
\{a_j\circ A:A\in\mathcal G\}
\]
is normal. Let $g,h:X\to\mathbb S^n$ be $K$-quasiregular mappings,
with $h$ normal, and suppose
\[
\E(g,a_j)=\E(h,a_j),
\qquad j=1,\ldots,q.
\]
If $q\geq q_0+1$, then $g$ is normal.
\end{theorem}

\begin{proof}
Let $(A_m)$ be an arbitrary sequence in $\mathcal G$. By a diagonal
subsequence argument,
\[
a_j\circ A_m\longrightarrow b_j
\]
locally uniformly in the spherical metric for every $j$. Since $h$ is
normal, after passing to a further subsequence,
\[
h\circ A_m\longrightarrow h_0.
\]
The uniform separation gives
\[
\distS(b_j(x),b_k(x))\geq\delta,
\]
so the limiting targets are pairwise distinct. Moreover,
\[
\E(g\circ A_m,a_j\circ A_m)
=
\E(h\circ A_m,a_j\circ A_m).
\]
Theorem~\ref{thm:shared-moving} applies to the corresponding local
sequences. Hence $(g\circ A_m)$ is normal locally on $X$.
Definition~\ref{D: NQM}(i) gives the conclusion.
\end{proof}

\begin{corollary}
Under the hypotheses of Theorem~\ref{T: sharing-NQM-S}, if
\[
a_j(x)\neq h(x)
\qquad
(x\in X,\ j=1,\ldots,q),
\]
then $q\geq q_0$ is sufficient.
\end{corollary}

\begin{proof}
Apply Corollary~\ref{cor:shared-separated} to the subsequences obtained
in the proof of Theorem~\ref{T: sharing-NQM-S}.
\end{proof}

\section*{Acknowledgements}

The third author gratefully acknowledges the
fellowship received from CSIR-HRDG, New Delhi, India,
through a Junior Research Fellowship (File No.\ 09/0045(21572)/2025-EMR-I).
The authors are also thankful to Prof.\ Sanjay Kumar Pant
for his valuable discussions and suggestions.

\section*{Declarations}

\textbf{Author Contributions:}
All authors contributed equally to this work.
All authors wrote the main manuscript and reviewed the final version.

\textbf{Data Availability:}
Data sharing is not applicable to this article,
as no datasets were generated or analysed during the current study.

\textbf{Ethical Approval:}
This article does not involve any studies with human participants
or animals and therefore ethical approval is not required.

\textbf{Conflict of Interest:}
The authors declare that there is no conflict of interest.

\textbf{Competing Interests:}
The authors declare no competing interests.

\end{document}